\documentclass[11pt]{amsart}

\usepackage{amssymb}
\usepackage{epsfig}  		
\usepackage{epic,eepic}       
\usepackage{tikz}
\usepackage{tikz-cd}
\usepackage{amsfonts}
\usepackage{amsthm}
\usepackage{amsmath}

\newtheorem{thm}{Theorem}[section]
\newtheorem{proposition}[thm]{Proposition}

\theoremstyle{definition}
\newtheorem{definition}[thm]{Definition}

\theoremstyle{remark}

\numberwithin{equation}{section}

\begin{document}

\title{Minimality of the $\mathcal D$-groupoid of symmetries of a projective structure}

\author{Guy Casale}
\address{Univ Rennes, CNRS, IRMAR-UMR 6625, F-35000 Rennes, France}
\email{guy.casale@univ-rennes1.fr}

\author{David Bl\'azquez-Sanz}
\address{Escuela de Matemáticas, Universidad Nacional de Colombia, Medell\'in}
\email{dblazquezs@unal.edu.co}

\author{Alejandro Arenas Tirado}
\address{Escuela de Matemáticas, Universidad Nacional de Colombia, Medell\'in}
\email{aarenast@unal.edu.co}

\begin{abstract}
    In this article we study Kummer's $\mathcal D$-groupoid, which is the groupoid of symmetries of a meromorphic projective structure. We give necessary and sufficient conditions for its minimality, in the sense of not having infinite sub-$\mathcal D$-groupoids. The condition that we find turns out to be equivalent to the strong minimality of the non-linear Schwarzian equation and the non-integrability by means of Liouvillian functions of the linear Schwarzian equation. \\
 
    \noindent\textbf{keywords}: $\mathcal{D}$-groupoid, Schwarzian equation, Schwarzian derivative, Strong minimality,  Symmetric Power, Lie groupoid. 
\end{abstract}

\maketitle

\tableofcontents

\section{Introduction}
The {\it Schwarzian derivative} of a meromorphic function $f \in \mathcal{K}$ (the field of meromorphic functions over an open subset of a Riemann surface) with respect to a coordinate $z$ is defined by the expression:

$$S_{z}(f) = \left(\frac{f''}{f'} \right)' - \frac{1}{2}\left(\frac{f'' }{f'} \right)^2  \ \textrm{ where }\  ' = \frac{d}{dz}$$
 
Particularly, the Schwarzian derivative classifies the functions related by the action of ${\rm PSL}_2(\mathbb C)$ on $\overline{\mathbb C}$: 

$$S_z(f) = S_z (g) \Longleftrightarrow \exists \sigma \in {\rm PSL}_2(\mathbb{C})\ \textit{s.t.}\ f = \sigma \circ g$$

Modular functions satisfy a differential equation of the form:

\begin{equation}\label{eq:es1}
S_{\tau}(\lambda(\tau)) + \lambda'(\tau)^2R(\lambda(\tau)) = 0,
\end{equation}

\noindent where $R(\lambda) \in \mathbb{C}(\lambda)^{\rm alg}$ is an algebraic function \cite{yoshida2013fuchsian}. K. Mahler proved that these function do not satisfy lower order equations \cite{Mahler}. The chain rule for the Schwarzian derivative allows us to invert the above equation obtaining a linear-Schwarzian equation that is geometrically equivalent:

\begin{equation}\label{eq:es2}
S_{\lambda}(\tau(\lambda)) = R(\lambda).    
\end{equation}

Even if this equation is not linear, it is classical fact explained in section 1.1 that this equation is linearizable.  

Differential algebraic properties of modular functions are relevant and have been recently studied in connection with model theory and hypertranscendence theory \cite{aslanyan2021ax}, \cite{klingler2016hyperbolic}, \cite{pila2013modular}, \cite{pila2016ax}. 
In particular, it was shown in  \cite{blazquez2020some}, \cite{casale2020axlindemannweierstrass} that, for any particular algebraic function $R(\lambda)\in \mathbb{C}(\lambda)^{\rm alg}$, the strong minimality of the equation \eqref{eq:es1}, is equivalent to the non-integrability of equation \eqref{eq:es2} by means of Liouvillian functions.

In this article we study the symmetries of differential equations \eqref{eq:es1} and \eqref{eq:es2} given by change of variable $(\lambda, \tau) \mapsto (\varphi(\lambda), \sigma(\tau))$ with $\sigma \in {\rm PSL}_2(\mathbb C)$ and $\varphi$ a local biholomorphism. The object that describes these symmetries is a $\mathcal D$-groupoid defined by a third order equation on $\varphi(\lambda)$ that was first studied by Kummer in relation with hypergeometric functions \cite{kummer}. It is given by the differential equation :

\begin{equation}\label{eq:kummer}
S_{\lambda}(\varphi(\lambda)) + \varphi'(\lambda)^2R(\varphi(\lambda)) = R(\lambda)
\end{equation}

whose solutions are stable by compositions. It is called Kummer groupoid.

Our main result (Theorem \ref{MR}) states that, if equation \eqref{eq:es2} has no Liouvillian solutions then the Kummmer groupoid \eqref{eq:kummer} is minimal in the sense of not having non-trivial sub-$\mathcal D$-groupoids of order greater than $0$. The criterium for the minimality of this $\mathcal D$-groupoid coincides with a recently discovered criterium for the strong minimality of \eqref{eq:es1}. As a consequence, we obtain (see Proposition \ref{prop_final} numerals 2 and 4):

\noindent {\bf Theorem.} \ {\it The Schwarz equation \eqref{eq:es1} is strongly minimal if and only if the Kummer groupoid \eqref{eq:kummer} is minimal.}

If there is a direct connection between the strong minimality of differential equations and minimality of their $\mathcal D$-groupoids of symmetries in some more general context remains unclear and is an interesting question. 

In subsection \ref{SS.11} we explain our take on Schwarzian equations using the jet spaces and invariant connections. The interested reader may find additional information in \cite{casale2020axlindemannweierstrass, blazquez2020some}. Subsection \ref{SS.12} contains some well known tools of Picard-Vessiot theory that we will use in the proof of our result, main sources are \cite{KOVACIC19863, crespoalgebraic, van2003galois}. Section \ref{S.2} is devoted to the definition of $\mathcal D$-groupoid and its linearization. Here we give a synthetic exposition with the purpose of just give the necessary tools to the reader. The main reference is \cite{malgrange2010pseudogroupes} and there is also a general presentation of the subject in \cite{blazquez2020malgrange}. Finally in section \ref{S.3} we go into the exploration of Kummer's groupoid. The computation of the equation of Kummer's groupoid and its linealization is classical, then we go into the proof of some preliminary results and then our main result.

\subsection{The Schwarzian equation as a ${\rm PSL}_2$-connection}\label{SS.11}

We consider an algebraic curve $X$ as a suitable ramified covering of $\bar{\mathbb C}$ so that the algebraic function $R(\lambda) \in \mathbb C(X)$ is seen as a rational function on $X$. We also remove from $X$ all the zeroes and poles of $d\lambda$ and poles of $R(\lambda)$ so that the vector field $\frac{d}{d\lambda}$ has neither zeroes nor poles on $X$ and $R(\lambda)$ is a regular function on $X$. 

We see equations \eqref{eq:es1} and \eqref{eq:es2} geometrically as foliations in the jet space. Thus, we consider 
$J = J_*^2(\overline{\mathbb C}, X)$ the variety of $2$-jets of local biholomorphisms from $\overline{\mathbb C}$ to $X$. An affine open subset of this jet space is 
$\mathbb C \times X \times \mathbb C^* \times \mathbb C$ where the $4$-tuple $(\tau, \lambda, \lambda_\tau, \lambda_{\tau\tau})$ represents the order $2$ development of a biholomorphim sending $\tau$ to $\lambda$ with derivatives $\lambda_\tau$, and $\lambda_{\tau\tau}$. Differential equation \eqref{eq:es1} is seen  geometrically as the equations of the integral curves of the foliation $\mathcal F = \langle D_{\tau} \rangle$ generated by the vector field $D_{\tau}$ in the jet space $J$:
$$D_{\tau} = \frac{\partial}{\partial \tau} + \lambda_\tau\frac{\partial}{\partial\lambda} + \lambda_{\tau\tau}\frac{\partial}{\partial \lambda_{\tau}} + 
\left(\frac{3}{2}\frac{\lambda_{\tau\tau}^2}{\lambda_{\tau}} + \lambda_{\tau}^3R(\lambda)\right)\frac{\partial}{\partial \lambda_{\tau\tau}}.$$
The integral curves of $\mathcal F$ are the graphs of the $2$-jet prolongations 
$$j^2\hat \lambda \colon \overline{\mathbb C} \dasharrow J, \quad  \tau \mapsto (\tau, \hat\lambda(\tau), \hat\lambda_\tau(\tau), \hat\lambda_{\tau\tau}(\tau))$$ of solutions $\hat\lambda(\tau)$ of \eqref{eq:es1}.

On the other hand, note that the inversion of biholomorphims gives an isomorphism $J \simeq J_*^2(X, \overline{\mathbb C})$ that corresponds just to the interchange between the dependent and independent variable $\lambda$ and $\tau$. We consider this isomorphism just as change of coordinates in $J$ giving rise to new coordinates $\lambda, \tau, \tau_{\lambda}, \tau_{\lambda\lambda}$ for which the same foliation $\mathcal F$ represents equation \eqref{eq:es2}.

A projective structure on $X$ is a maximal atlas $\{(\mathcal U_i, \tau_i)\}_{i\in I}$ of coordinates $\tau_i \colon X \to \overline{\mathbb C}$ with the property that transition functions in $\overline{\mathbb C}$ are elements of ${\rm PSL}_2(\mathbb C)$. An important feature about equation \eqref{eq:es2} is that it defines a projective structure on $X$. Note that, if $\tau(\lambda)$ is a solution of \eqref{eq:es2} then any other solution with the same domain of definition is of the form $g(\tau(\lambda))$ for some $g\in {\rm PSL}_2(\mathbb C)$.

This structure of the solution space of \eqref{eq:es2} is reflected in the jet space. The action of ${\rm PSL_2}$ on $\overline{\mathbb C}$ lifts to an action on $J \simeq J_*^2(X, \overline{\mathbb C})$. The natural projection $\pi\colon J\to X$ is a principal bundle with structure group ${\rm PSL}_2$. 
 
The foliation $\langle D_{\tau} \rangle$ turns out to be a ${\rm PSL}_2$-invariant connection. Therefore, the theory of strongly normal extensions can be applied to the equation \eqref{eq:es2}; its Galois group will be an algebraic subgroup of ${\rm PSL}_2$.

Indeed, there is a well known explicit relation between the linear-Schwarzian equation \eqref{eq:es2} and the second order linear differential equation,

\begin{equation}\label{eq:linear}
\frac{d^2\psi}{d\lambda^2} = -\frac{1}{2}R(\lambda)\psi.
\end{equation}

Namely, the quotient $\tau = \psi_1/\psi_2$ between two any linearly independent solutions $\psi_1$, $\psi_2$ or \eqref{eq:linear} is a solution of \eqref{eq:es2}. This relation can be seen geometrically as an equivariant $2$-cover of principal bundles,

$${\rm SL}_2(\mathbb C) \times X \to J $$
$$\left(
\left[
\begin{array}{cc} a & b \\ c & e \end{array}
\right], \lambda
\right)  \mapsto j^2_\lambda  \left( \frac{a\tau + b}{c\tau + e}\right)$$
that maps the companion system

\begin{equation}\label{eq:companion}
\frac{d}{d\lambda}
\left[
\begin{array}{cc} \psi_{1} & \psi_{2} \\ \psi_{1}' & \psi_{2}' \end{array}
\right] = \left[
\begin{array}{cc} 0 & 1 \\ -\frac{1}{2}R(\lambda) & 0 \end{array}
\right]
\left[
\begin{array}{cc} \psi_{1} & \psi_{2} \\ \psi_{1}' & \psi_{2}' \end{array}
\right] 
\end{equation}

of equation \eqref{eq:linear} to the linear-Schwarzian equation \eqref{eq:es2}. Note that this map is well defined except on the singularities of $d\lambda$ that we already removed from our algebraic curve $X$.
In \cite{blazquez2020some} and \cite{casale2020axlindemannweierstrass} it is shown that the equation \eqref{eq:es1} is strongly minimal if and only if the Galois group of the equation \eqref{eq:linear} is exactly ${\rm SL}_2$. This last condition is equivalent to the non-integrability  of either \eqref{eq:es2} or \eqref{eq:linear} by means of Liouvillian functions.

\subsection{Some relevant facts of Picard-Vessiot theory}\label{SS.12}

Let us discuss here some facts around the integrability of equations \eqref{eq:es2} and \eqref{eq:linear} in the context of Picard-Vessiot theory, which is the part of differential Galois theory that deals with linear differential equations.  The interested reader may consult section 7.3 of \cite{crespoalgebraic} or \cite{KOVACIC19863} for the original source. There is also slightly more extended exposition of these facts in \cite{blazquez2020some}.
In our case the field of coefficients is $\mathcal K = \mathbb C(X)$ endowed of the derivation $\frac{d}{d\lambda}$. All the facts explained in this section hold for any differential field of characteristic zero with algebraically closed field constants.

Let us consider the differential field extension $\mathcal K \subseteq \mathcal K\langle \psi_1,\psi_2\rangle$ spanned by two linearly independent solutions of \eqref{eq:linear}. The differential Galois group of equation \eqref{eq:linear} is the group of differential field automorphisms $ G = {\rm Aut}(\mathcal K\langle \psi_1,\psi_2\rangle / \mathcal K)$. This group is naturally represented as an algebraic subgroup of ${\rm SL}_2(\mathbb C)$. For each $\sigma \in G$ we have,

\begin{equation}\label{sigma}
\left[
\begin{array}{cc} \sigma(\psi_{1}) & \sigma(\psi_{2}) \end{array}
\right] = 
\left[
\begin{array}{cc} \psi_{1} & \psi_{2} \end{array}
\right].
\left[
\begin{array}{cc} a_{\sigma}  & c_{\sigma} \\ b_{\sigma} & e_{\sigma} \end{array}
\right]
\end{equation}

The fundamental theory of integrability by Liouvillian functions in differential Galois theory says that linear differential equation as \eqref{eq:es2} is integrable by mean of Liouvillian functions if and only if the Lie algebra of $G$ is solvable. On the other hand, all proper subgroups of ${\rm SL}_2(\mathbb C)$ have solvable Lie algebra. Eigenvectors of the action of Lie algebra of $G$ on the vector space of solutions are related with algebraic solutions of the auxiliar Riccati equation, 

\begin{equation} \label{riccati1}
    u' + u^2 +\frac{1}{2}R(\lambda)= 0
\end{equation}

satisfied by $u = \frac{d \log \psi}{d\lambda}$, the logarithmic derivative of a solution of \eqref{eq:linear}. The following proposition accounts for the preliminary considerations before Kovacic's algorithm in \cite{KOVACIC19863} and are well known facts in the context of Picard-Vessiot theory.

\begin{proposition}\label{T:44}
Let us consider the differential equation \eqref{eq:linear}. There are the following four mutually exclusive possibilities for the solutions:
\begin{description}
    \item[(Case 1)] The Ricatti equation \eqref{riccati1} has at least a solution $u\in \mathcal K$.  $\psi = e^{\int u}$ is a Liouvillian solution of
    \ref{eq:linear} and the Galois group $G$ is conjugated to a group of triangular matrices.
    \item[(Case 2)]  The Ricatti equation \eqref{riccati1} has a pair of conjugated solutions $u_{\pm}$ that are algebraic of degree $2$ over $\mathcal K$. 
    $\psi_{\pm} = e^{\int u_{\pm}}$, are algebraically independent Liouvillian solutions of \eqref{eq:linear} and the Galois group $G$ is conjugated to a subgroup of the infinite dihedral group. 
    \item[(Case 3)] All solutions of \eqref{eq:linear} are algebraic over $\mathcal K$ and the Galois group $G$ is conjugated to a finite crystallographic group.
    \item[(Case 4)] Equation \eqref{eq:linear} has no Liouvillian solution and the Galois group is $G = {\rm SL}_2(\mathbb C)$. 
\end{description} 
\end{proposition}

Let us consider now $\tau = \psi_1 / \psi_2$, which is a solution of \eqref{eq:es2}, and the tower of differential fields,
$$\mathcal K \subseteq \mathcal K\langle \tau \rangle \subseteq K\langle \psi_1,\psi_2 \rangle.$$
From equation \eqref{sigma} we have that,
$$\sigma(\tau) = \frac{a_\sigma \tau + b_\sigma}{c_\sigma \tau + d_\sigma} \in \mathbb C\langle \tau \rangle$$
It follows that $\mathcal K \subset\mathcal K\langle \tau \rangle$ is a Picard-Vessiot extension whose differential Galois group is the quotient of $G$
by the stabilizer of $\tau$,
$${\rm Aut}(\mathcal K\langle \tau \rangle / \mathcal K) = \overline G = G / (G \cap \{ {\rm I}, -{\rm I}\}).$$
Note that cases 1, 2, 3 of Proposition \ref{T:44} will lead us to a Liouvillian extension $\mathcal K \subseteq \mathcal K\langle \tau \rangle$ and $\overline G$ a proper subgroup of ${\rm PSL}_2(\mathbb C)$ and case 4 implies the non-existence of Liouvillian solutions for \eqref{eq:es2} and $\overline G = {\rm PSL}_2(\mathbb C)$.

\section{$\mathcal{D}$-Groupoids}\label{S.2}

The notion of $\mathcal D$-groupoid was introduced by B. Malgrange in \cite{malgrange2001} in the context of non-linear differential Galois theory. We may also call then algebraic Lie pseudogroups, as they are systems of algebraic differential equations whose solutions form a Lie pseudogroup. The original proposal allowed differential equations that where analytic in the base and algebraic on the derivatives, but later formulations restricted the definition to algebraic differential equations. Here, we will give a definition  that is equivalent to definition 5.2 in \cite{malgrange2010pseudogroupes} (Definition \ref{def_groupoid}); the equivalence between these definitions can be found in appendix A of \cite{blazquez2020malgrange}.

From this part on, in order to make the notation simple the derivative symbol $f'$ will be used with the meaning of the derivative with respect to $\lambda$.

\subsection{Jets of biholomorphisms}
\label{S.2.1}

The jet space $J^k(X,X)$ is defined as the set of equivalence classes of contact of order $\geq k$ at points of $X$ of local biholomorphisms from open subsets of $X$ to open subsets of $X$.\footnote{To be more precise, in general, taking a $\pi \colon E \to M$ a submersion between varieties; $\varphi_1$, $\varphi_2$ local sections defined around $p\in M$ have contact of order $\geq k$ if $d\varphi_1$, and $d\varphi_2$ have contact of order $\geq k-1$, that is, $\varphi_1(p)=\varphi_2(p),\cdots, d_p(d^{k-1}\varphi_1)=d_p(d^{k-1}\varphi_2)$. A local holomorphism is seen as a local section of the trivial bundle $\pi_1 \colon X\times X \to X$.} This space
$J^k(X,X)$ has a natural structure of algebraic variety (see, for instance \cite{blazquez2020malgrange}), and it contains the open subset of jets of biholomorphism:
$${\rm Aut}_k(X) := J_*^k(X,X)$$
Jets of biholomorphisms may be composed and inverted taking into account their sources and targets, so that ${\rm Aut}_k(X)$ is an algebraic groupoid over $X$;
the groupoid of $k$-jets of invertible local biholomorphisms on $X$. 

In order to clarify the algebraic structure of ${\rm Aut}_k(X)$ let us examine the case in which $X\subseteq \mathbb C$ is an affine subset of the complex numbers, with coordinate $\lambda$. The general case is recovered by gluing coverings of that case. 

A $(k+2)$-tuple $( \lambda_0, \varphi_0, \varphi'_0,...,\varphi^{(k)}_0)$ corresponds to the $k$-jet in $\lambda_0$ of any biholomorphism of the form:

$$\varphi: \lambda \mapsto \varphi_0+\varphi_0'(\lambda-\lambda_0)+\frac{\varphi_0''}{2}(\lambda-\lambda_0)^2+ \cdots+\frac{\varphi_0^{(k)}}{k!}(\lambda-\lambda_0)^k+o(\lambda-\lambda_0)^{k+1}$$

From now on we do not mention the subindex zero, so $\lambda$, $\varphi$, $\varphi'$, $\ldots$, $\varphi^{(k)}$ is a system of coordinates in ${\rm Aut}_k(X) \simeq X \times X \times \mathbb C^* \times \mathbb C^{k-1}$. The same direct product decomposition is possible for an affine algebraic curve $X$, providing that the vector field $\frac{d}{d\lambda}$ has neither zeros or poles in $X$.  The composition law in ${\rm Aut}_k(X)$ is given by Fa\`a di Bruno formulae, and thus, it is polynomial in the coordinates of ${\rm Aut}_k(X)$ wich turns out to be an algebraic groupoid over $X$. Truncations ${\rm Aut}_k(X)\to {\rm Aut}_{k-1}(X)$ are compatible with composition and inversion. Hence, the projective limit ${\rm Aut}(X):= \varprojlim {\rm Aut}_k(X)$ inherits the groupoid structure. Elements of ${\rm Aut}(X)$ are formal local biholomorphisms, not necessarily convergent. 

\begin{definition} The following types of algebraic subvarieties of ${\rm Aut}_k (X)$ will be considered.
\begin{enumerate}
    \item A strict algebraic subgroupoid of ${\rm Aut}_k(X)$ is a Zariski closed subset $Z \subset {\rm Aut}_k(X)$ which is also a smooth subgroupoid.
    \item A rational subgroupoid $\mathcal{G}_k \subset {\rm Aut}_k(X)$ is a Zariski closed subset such that on an open $U \subset X$,  $\mathcal G_k|_U$ is dense in $\mathcal G_k$ and is a strict algebraic subgroupoid of ${\rm Aut}_k(U)$.
\end{enumerate}
\end{definition}

\subsection{Zariski Topology of ${\rm Aut}(X)$}

Let us consider the case in which $X$ is an affine curve and the vector field $\frac{d}{d\lambda}$ has neither zeroes or poles, so that
${\rm Aut}_k(X) \simeq X \times X \times \mathbb C^* \times \mathbb C^{k-1}$ is an affine algebraic manifold with ring of regular functions $\mathcal O_{{\rm Aut}_k(X)}$. Taking limit $k\to\infty$ we obtain the proalgebraic variety ${\rm Aut}(X)$ with ring of regular functions,
$$\mathcal O_{{\rm Aut}(X)} = \bigcup_k \mathcal O_{{\rm Aut}_k(X)}.$$
The general case is similar but instead of 
ideals of the ring of regular functions $\mathcal O_{{\rm Aut}_k(X)}$ we would be forced to use coherent sheaves of ideals of its structural sheaf.

A \textit{Zariski closed} subset $Z$ of ${\rm Aut}(X)$ is defined by a radical ideal $\mathcal E \subset \mathcal O_{{\rm Aut}(X)}$.
The $k$-order truncation of $Z$ is the Zariski closed $Z_k$ of ${\rm Aut}_k(X)$, defined by the ideal $\mathcal E_k = \mathcal E \cap \mathcal{O}_{{\rm Aut}_k(X)}$, $Z_k:= V(\mathcal{E}_k)$.  In this way, $Z$ can be viewed as a sequence of closed Zariski sets $\{Z_k\}_{k\in \mathbb N}$ such that each projection $Z_k \to Z_{k-1}$ is dominant, that means
$$\cdots \longrightarrow Z_k \longrightarrow Z_{k-1} \longrightarrow \cdots \longrightarrow Z_0 \subset X \times X$$ is a dominant sequence.

\subsection{Kolchin Topology and definition of $\mathcal D$-groupoid}

The total derivative is a canonical way of extending derivations in $\mathcal O_X$ to derivations of $\mathcal{O}_{{\rm Aut}(X)}$.
Given a vector field $\vec{w}$ in  $X$, we define its total derivative,  

$$\vec{w}^{\textit{tot}}\colon \mathcal{O}_{{\rm Aut}_k(X)}\to \mathcal{O}_{{\rm Aut}_{k+1}(X)}, \quad
(\vec w^{tot}f)(j_x^{k+1}\varphi)=\vec w_x(f \circ j^k\varphi).$$

Let us recall that $\frac{d}{d\lambda}$ is a vector field without zeros nor poles in $X$. Through the total derivation mechanism it extends to a derivation of $\mathcal{O}_{{\rm Aut}(X)}$ that we denote by the same symbol (without the ${}^{tot}$ superscript), so that $\mathcal O_{\rm Aut(X)}$, endowed with 
$$\frac{d}{d\lambda}  = \frac{\partial}{\partial \lambda}  + \varphi' \frac{\partial}{\partial \varphi} + \varphi''\frac{\partial}{\partial \varphi'} + \ldots$$ 
is a differential ring.

An ideal $\mathcal{J}$ of $\mathcal{O}_{{\rm Aut}(X)}$ is an \textit{$\mathcal{D}$-ideal} if for every differential function $f$ in $\mathcal{J}$ the total derivatives of $f$ are also in $\mathcal{J}$. A subset $Y \subset {\rm Aut}(X)$ is a \textit{Kolchin closed} if it is a Zariski closed, whose ideal is a $\mathcal{D}$-ideal. That is, a closed set in Kolchin topology is given by the zeros of a radical $\mathcal{D}$-ideal.

\begin{definition}\label{def_groupoid}
A $\mathcal{D}$-groupoid $\mathcal{G}= \{\mathcal{G}_k\}_{k\in \mathbb{N}} \subset {\rm Aut}(X)$ is a Kolchin closed set such that for all $k$, $\mathcal{G}_k$ is a rational subgroupoid. The smaller $k$ such that $\mathcal G_k \subset {\rm Aut}_k(X)$ is called the order of $\mathcal G$.
\end{definition}

\begin{enumerate}
\item[a)] Given a $\mathcal D$-groupoid $\mathcal G$ then there is an open set $\mathcal U \subset X$ such that $\mathcal G|_{\mathcal U}$ is a groupoid over $\mathcal U$, see \cite{blazquez2020malgrange} appendix A.
\item[b)] Solutions of $\mathcal G$, meaning, local biholomorphisms $f\colon X\dasharrow X$\footnote{The dash arrow is used to denote maps defined on a open subset of $X$} such that for all $x$ in its domain $j_xf \in \mathcal G$ form a pseudogroup of transformations of $X$.
\item[c)] In the general setting $\mathcal G$ is defined by a system of algebraic PDE with as many independent variables and unknowns as the dimension of the base variety. As $X$ is one-dimensional then $\mathcal G$ is defined by ODE. Let us consider $\mathcal G$ of order $k$. As $\mathcal G_k$ dominates ${\rm Aut}_{k-1}(X)$ it must be an hypersurface of ${\rm Aut}_{k}(X)$. As ${\rm Aut}_{k}(X)$ is affine then the ideal of $\mathcal G_k$ is spanned  single 
element $F(\lambda, \varphi, \ldots, \varphi^{(k)})\in {\mathcal O}_{{\rm Aut}_k(X)}$. From the differential equation $F = 0$ we deduce,

$$\frac{d}{d\lambda} F =  \varphi^{(k+1)}\frac{\partial F}{\partial \varphi^{(k)}} + Q  = 0 ; \quad \varphi^{(k+1)} = -\frac{Q}{F_{\varphi^{(k)}}}$$

a $(k+1)$-order differential equation where the $(k+1)$-th derivative is expressed as a rational function of the lower degree derivatives, and same for higher order derivatives. Summarizing, a $\mathcal D$-groupoid of order $k$ on an affine algebraic curve $X$ with a non vanishing vector field is always determined by a single $k$-th order differential equation. 
\end{enumerate}

\subsection{$\mathcal D$-algebra of a $\mathcal D$-groupoid}

Consider the tangent bundle $TX \to X$, and define  ${\rm aut}_k(X) = J^k(TX/X)$ as the bundle of $k$-jets of sections of the tangent bundle. The vector bundle ${\rm aut}_k(X)\to X$ with its anchor ${\rm aut}_k(X)\to {\rm aut}_0(X) = TX$ is the Lie algebroid of ${\rm Aut}_k(X)$. Let us explore how elements of ${\rm aut}_k(X)$ are identified with vectors tangent to the identity in ${\rm Aut}_k(X)$. For this, consider, $j^k_x\vec{w}^{(k)} \in {\rm aut}_{k}(X)$ so that 

$$f(\lambda) = f(x)+f'(x)(\lambda-\lambda(x))+f''(x)\frac{(\lambda-\lambda(x))^2}{2}+\cdots$$

For $\varepsilon$ varying in $\mathbb C$, $j^k_x(\varepsilon \vec w)$ is a curve on ${\rm aut}_k(X)$.
If $\varepsilon$ is sufficiently small, The existence theorem for ordinary differential equations implies that the exponential 
$${\rm exp}(\varepsilon \vec w) \colon (X,x) \to (X,y(\varepsilon))$$ 
is an analytic map, where  $y(\varepsilon) = {\rm exp}(\varepsilon \vec w)x$. Thus, by taking the $k$-jet of exponential $j^k_x {\rm exp}(\varepsilon \vec{w})$ as a curve in ${\rm Aut}_k(X)$,
$$\Phi_k \colon {\rm aut}_k(X) \hookrightarrow T({\rm Aut}_k(X)), \quad
\vec w \mapsto \left.\frac{d}{d\varepsilon}\right|_{\varepsilon =0}j_x^k({\rm exp}(\epsilon \vec{w})),$$
so that each element of ${\rm aut}_k(X)$ is seen a a vector tangent to the identity in ${\rm Aut}_k(X)$ preserving 
the source map. Taking the projective limit $k\to\infty$ we obtain,
$$\Phi\colon {\rm aut(X)} \hookrightarrow T({\rm Aut}(X))$$
whose expression in coordinates is 
$$j_x \left(f(\lambda) \frac{\partial}{\partial \lambda}\right) \mapsto f(x) \frac{\partial}{\partial \varphi} + f'(x) \frac{\partial}{\partial \varphi_{\lambda}} + f''(x) \frac{\partial}{\partial \varphi_{\lambda\lambda}} + \ldots$$

Now, we are studying certain $\mathcal D$-groupoid $\mathcal G\subset {\rm Aut}(X)$. We may ask for which germs vector fields $j_x \vec w \in {\rm aut}(X)$ have exponential $j_x (\varepsilon \vec w)$ in $\mathcal G$. This happens if $\Phi(j_x \vec w)$ is a vector tangent to $\mathcal G$. 
This is the problem of linearizing the groupoid $\mathcal G$.

Consider some differential equation

$$F(\lambda,\varphi,\varphi',\varphi'',\ldots,\varphi^{(k)})=0$$

which vanish on the $\mathcal{G}$. Here $F\in {\mathcal O}_{{\rm Aut}(X)}$. We linearize $F$ by taking,

$$\ell F\left(j_x^k f(\lambda)\frac{\partial}{\partial \lambda} \right) = dF\left(\Phi_k\left(j_x^k f(\lambda)\frac{\partial}{\partial \lambda} \right)\right)$$

so that $\ell F \in {\mathcal O}_{{\rm aut}(X)}$ is in fact a linear form in the coefficients of the power series development of
$f(\lambda)$ and thus a linear differential equation for the unknown $f$.

We can now define ${\rm Lie}(\mathcal G_k) \subset {\rm aut}_k(X)$ as the rational linear bundle defined by all the linearizations $\ell F$ of functions vanishing on $\mathcal G_k$, and the $\mathcal D$-Lie algebra of $\mathcal G$ as the projective limit of this system ${\rm Lie}(\mathcal G) \subset {\rm aut}(X)$. It is a system of linear differential equations whose solutions are vector fields in $X$. Without discussing the structure of these objects, let us take note of the two following relevant facts:

\begin{enumerate}
    \item[a)] The Lie bracket of two solutions of ${\rm Lie}(\mathcal G)$, when defined, is also a solution of  ${\rm Lie}(\mathcal G)$.
    \item[b)] If $\mathcal G$ is determined by single differential equation of order $k$ then ${\rm Lie}(\mathcal G)$ is determined by a single linear differential equation of order $k$.
\end{enumerate}

More theoretical results on $\mathcal D$-Lie algebras in relation with $\mathcal D$-Lie groupoid can be found in \cite{malgrange2010pseudogroupes}.

\section{Kummer's groupoid}\label{S.3}

Let us consider $\varphi \colon X \dasharrow X$ an local biholomorphism. We want to check if $\sigma$ is compatible with the projective
structure induced by equation \eqref{eq:es2} in the sense that composition with $\sigma$ sends local projective charts on local projective
charts of the same structure. In other word, for each solution $\tau$ of \eqref{eq:es2} the composition $\tau\circ \sigma$ is also a solution. 
From the chain rule for the Schwarzian derivative we obtain:

$$S_{\lambda}(\tau \circ \varphi) = (S_\lambda (\tau)\circ \varphi)\varphi_\lambda^2+S_\lambda(\varphi)$$

Now, from equation \eqref{eq:es2} after composition with $\varphi$ we obtain $S_\lambda(\tau)\circ \varphi = R(\varphi)$. Finally we obtain a differential equation for $\varphi$ as function of $\lambda$:

 \begin{equation} \label{kummer}
    S_\lambda(\varphi) = R(\lambda)-R(\varphi)\varphi_\lambda^2.
\end{equation}

The above differential equation characterizes the symmetries of \eqref{eq:es1} and then it defines a $\mathcal D$-groupoid of $X$ which is the Kolchin closed subset $\mathcal G \subset {\rm Aut}(X)$ determined by the radical $\mathcal D$-ideal generated by $ S_{\lambda}(\varphi) - R(\lambda) + R(\varphi)\varphi_{\lambda}^2$. 
As this equation was first presented by Kummer in \cite{kummer}), we refer to $\mathcal G$ as Kummer's groupoid. This equation also appear in various different work during the last century : in Ritt's work on hypertranscendency of Koenings's linearisations \cite{ritt}, in Ecalle synthesis of binary parabolic diffeomorphisma \cite{ecalle}, or in the classification of rational transformations of $\mathbb{CP}_1$ preserving an rational geometric structure  \cite{casaleP1}.

In ${\rm Aut}_3(X)$, the rational subgroupoid  $\mathcal G_3$ consisting of $3$-jets of solutions of the Kummer equation is:
$$\mathcal{G}_3:= \{j_{\lambda}^3\varphi \vert \varphi\colon \lambda \mapsto \varphi(\lambda),\,\, R(\varphi)\varphi_{\lambda}^2+S_{\lambda}(\varphi)=R(\lambda)\}$$
It corresponds to the $3$-jets of local biholomorphisms of $X$ that are compatible with the projective structure. Note that $\mathcal G$ is defined by a single differential equation of order $3$, which is seen as a function on ${\rm Aut}_3(X)$. By applying total derivative with respect to $\lambda$ we obtain that all derivatives of order higher than two can be written as a function of $\lambda, \varphi,\varphi_{\lambda},\varphi_{\lambda\lambda}$ and therefore:
$$\mathcal G \simeq \mathcal G_k \simeq \ldots  \mathcal G_3 \simeq \mathcal G_2 = {\rm Aut}_2(X).$$
as algebraic varieties. Hence $\mathcal{G}$ is an algebraic Lie groupoid of complex dimension 4 and it is isomorphic, as Lie groupoid, to ${\rm Aut}_2(X)$.

\subsection{Linearization of the Kummer's equation}

Let us write the Kummer equation \eqref{kummer} 

$$R(\varphi)\varphi'^2+S-R(\lambda)=0$$

where the Schwarzian derivative is now seen as a function of the coordinates
in ${\rm Aut}_3(X)$

   $$S(\lambda,\varphi,\varphi',\varphi'',\varphi''') =\frac{\varphi'''}{\varphi'}-\frac{3}{2}\left( \frac{\varphi''}{\varphi'}\right)^2.$$
   
Let us consider $\vec w = f(\lambda)\frac{\partial}{\partial \lambda}$. We obtain by direct computation:
\begin{align*}
    \frac{\partial S}{\partial \varphi} &= 0; \\
    \frac{\partial S}{\partial \varphi'} &= -\frac{\varphi'''}{\varphi'^2}+2\frac{\varphi''}{\varphi'}\frac{\varphi''}{\varphi'^2} \equiv 0 \quad (\mbox{along }\mathcal G ); \\
    \frac{\partial S}{\partial \varphi''} &= -3\frac{\varphi''}{\varphi'^2} \equiv 0 \quad(\mbox{along }\mathcal G );\\
    \frac{\partial S}{\partial \varphi'''} &= \frac{1}{\varphi'} \equiv 1 \quad(\mbox{along }\mathcal G ); \\
    \Phi(j_x\vec{w}) S &= f'''(x); \\
    \Phi(j_x\vec{w}) R(\lambda) &= 0; \\
    \Phi(j_x\vec{w}) R(\varphi)\varphi'^2 &= f(x)R'(\lambda(x))\lambda'^2+2R(\lambda(x))f'(x)\lambda' 
   \\
   & \equiv f(x)R'(\lambda) +2R(\lambda(x))f'(x) \quad(\mbox{along }\mathcal G )
\end{align*}

And thus, the linearized differential equation is

\begin{equation}\label{linear:kummer}
f''' +2R(\lambda)f' +R'(\lambda)f=0.
\end{equation}

It provides the necessary and sufficient conditions for the vector field  $f(\lambda)\frac{\partial}{\partial \lambda}$ to be tangent to $\mathcal G$.
This linear differential equation, defines a linear and closed Kolchin sub-bundle ${\rm Lie}(\mathcal G)$ of  ${\rm aut}(X)$, known as \emph{$\mathcal{D}$-Lie algebra of the Kummer's groupoid} $\mathcal{G}$. \\

\subsection{Symmetric Power}

In this section we show the relation between the equation \eqref{linear:kummer} of ${\rm Lie}(\mathcal G)$ and the the Riccati equation
\eqref{riccati1}. For this purpose, we need to introduce the symmetric power of the second order linear differential equation \eqref{eq:linear}. As before,
let $\mathcal K$ the differential field consisting of $\mathbb C(X)$ endowed with $\frac{d}{d\lambda}$, and lat $\psi_1$, $\psi_2$ be a pair of linearly independent solutions, so that we have the following extensions of differential fields.

\[\begin{tikzcd}
	\mathcal K && {\mathcal K\langle\psi_1,\psi_2\rangle} \\
	& {\mathcal K\langle\psi_1\psi_2,\psi_1^2,\psi_2^2\rangle}
	\arrow[shorten <=8pt, shorten >=8pt, hook, from=1-1, to=1-3]
	\arrow[shorten <=3pt, shorten >=3pt, hook, from=1-1, to=2-2]
	\arrow[shorten <=3pt, shorten >=3pt, hook, from=2-2, to=1-3]
\end{tikzcd}\]

The three functions $\psi_i\psi_j$ for $i,j = 1,2$ generate the solution space of a third order equation, which is known as the  \emph{second symmetric power}. To derive it, consider $f = \psi^2$, where $\psi = a\psi_1 + b\psi_2$ is any solution of \eqref{eq:linear}, we have:

\begin{align*}
    f' &= 2\psi \psi' \\
    f'' &= 2\psi'^2+2\psi\psi''=2\psi'^2-R(\lambda)a \\
    f''' &= 4\psi'\psi'' -R'(\lambda)a-R(\lambda)a' \\
        &= 4\psi'\left(-\frac{1}{2}R(w)\psi\right) -R'(\lambda)a-R(\lambda)a' \\
        &= -2\psi\psi'R(\lambda) -R'(\lambda)a-R(\lambda)a' \\
        &= -2R(\lambda)a' -R'(\lambda)a
\end{align*}
So, rewriting

\begin{equation} \label{psim}
f''' + 2R(\lambda)f'+ R'(\lambda)f = 0    
\end{equation} 

Which is the same third order differential equation that \eqref{linear:kummer}. In this way, we say that the third order linear differential equation \eqref{linear:kummer} is the second symmetric power of the second order linear equation \eqref{eq:linear}. From the above diagram, we have that $\mathcal K \subset \mathcal K\langle \psi_1^2, \psi_2^2, \psi_1\psi_2\rangle$ is the Picard-Vessiot extension of the third order equation \eqref{psim}.

\begin{proposition} 
The Galois groups ${\rm Aut}(\mathcal K\langle \psi_1^2,\psi_2^2,\psi_1\psi_2 \rangle/\mathcal K)$ and ${\rm Aut}(\mathcal K\langle \psi_1,\psi_2 \rangle/\mathcal K)$ have isomorphic Lie algebras.
\end{proposition}

\begin{proof}
Let us examine the diagram of extensions of differential fields, with $G =  {\rm Aut}(\mathcal K\langle \psi_1,\psi_2 \rangle/\mathcal K)$, 
$H = {\rm Aut}(\mathcal K\langle \psi_1,\psi_2 \rangle/\mathcal K\langle \psi_1^2,\psi_2^2,\psi_1\psi_2 \rangle)$ so that, by Galois correspondence
we have that  ${\rm Aut}(\mathcal K\langle \psi_1,\psi_2 \rangle/\mathcal K)\simeq G/H$.
\[\begin{tikzcd}
	K && {K\langle \psi_1,\psi_2 \rangle} \\
	& {K\langle\psi_1^2,\psi_2^2,\psi_1\psi_2\rangle}
	\arrow["G", hook, from=1-1, to=1-3]
	\arrow["{G/H}"', from=1-1, to=2-2]
	\arrow["H"', from=2-2, to=1-3]
\end{tikzcd}\]
It suffices to observe that the second inclusion is either the identity or an algebraic extension of finite degree two. By the Galois correspondence, the Galois group of the small extension is a quotient of that of the large extension by a finite normal subgroup and therefore has the same Lie algebra. 
\end{proof}

\begin{proposition}\label{prop:equivalence}
The following statements are equivalent:     
\begin{itemize}
    \item[(a)] ${\rm Aut}(\mathcal K\langle \psi_1,\psi_2\rangle / \mathcal K) = {\rm SL}_2$.
    \item[(b)] ${\rm Aut}(\mathcal K\langle \psi^2\rangle / \mathcal K) = {\rm PSL}_2$.
    \item[(c)] The Riccati equation \eqref{riccati1} has no algebraic solution over $\mathcal K$.
\end{itemize}
\end{proposition}

\begin{proof}
From the Galois correspondence \cite[Proposition 1.34 p. 25]{van2003galois} applied to the above diagram  (a)$\Longleftrightarrow$(b), since ${\rm PSL}_2$ is the only possible quotient of ${\rm SL}_2$ by a finite group. The equivalence (a)$\Longleftrightarrow$(c) is consequence of Proposition \ref{T:44} (see \cite[Excercise 1.36 page 28]{van2003galois}).
\end{proof}

We can now proof the main result of this article. 

\begin{thm}\label{MR}
If the Ricatti equation $\eqref{riccati1}$ has no algebraic solutions, then the Kummer  $\mathcal D$-groupoid $\mathcal G$ has no sub-$\mathcal D$-groupoids of order greater than $0$.
\end{thm}
\begin{proof}
    Assume that $\mathcal H \subseteq\mathcal G$ is a sub-$\mathcal D$-groupoid of order greater than $0$. Then the order of $\mathcal H$ is $1$, $2$ or $3$.
    Let us examine those cases separately. 
    \begin{enumerate}
    \item[a)] Let us assume that $\mathcal H$ is of order $3$. The equation \eqref{kummer} of $\mathcal G$ allows to express $\sigma'''$ as a rational function of $\lambda,\sigma,\sigma',\sigma''$. Therefore $\mathcal G_3$ is the graph of a section from ${\rm Aut}_2(X)$ to ${\rm Aut}_3(X)$. Therefore, it is an irreducible surface. As $\mathcal H$ is also an hypersuface of ${\rm Aut}_3(X)$ and it is contained in $\mathcal G_3$ then we have $\mathcal H_3 = \mathcal G_3$ and then $\mathcal H = \mathcal G$.
    
    \item[b)] Let us assume that $\mathcal H$ is of order $2$. Then ${\rm Lie}(\mathcal H)$ is determined by a second order differential equation
    with coefficients in $\mathcal K = \mathbb C(X)$.
    
    \begin{equation}\label{eq:aux}
    f'' + \alpha(\lambda) f' + \beta(\lambda) f = 0
    \end{equation}
    
    As ${\rm Lie}(\mathcal H) \subset {\rm Lie}(\mathcal G)$ then, the $3$-dimensional solution space of the equation \eqref{linear:kummer} contains 
    two linearly independent solutions of \eqref{eq:aux}. Note that this $3$-dimensional space is spanned by $\psi_1^2$, $\psi_2^2$ and $\psi_1\psi_2$ where $\psi_1$ and $\psi_2$ are linearly independent solutions of \eqref{eq:linear}. Inside this $3$-dimensional space, the set of elements of the form $\psi^2$ with $\psi$ a solution of \eqref{eq:linear} form a cone. In a complex $3$-dimensional vector space every plane intersects a cone on $2$ lines or a double line. Thus, there is a solution $f = \psi^2$ of \eqref{eq:aux} which is a square of a solution of \eqref{eq:linear}. Replacing and using the original equation,
    
    $$(\beta(\lambda) - R(\lambda))\psi^2 + 2\alpha(\lambda) \psi\psi' + 2\psi'^2 = 0$$
    
    \begin{enumerate}
    \item[b.i)] If $\beta(\lambda) = R(\lambda)$, we have $2\psi'(\psi\alpha(\lambda) - \psi') = 0$.
          In that case, we also have that either $\psi$ is constant or $\psi'/\psi = \alpha(\lambda)$ is a solution of the Riccati equation \eqref{riccati1}. Both cases contradict the hypothesis.

    \item[b.ii)] If  $\beta(\lambda) = R(\lambda)$, there are algebraic functions of degree one or two such that: 
    $$(\beta(\lambda) - R)(\psi - \gamma_1(\lambda)\psi')(\psi - \gamma_2(\lambda)\psi') = 0$$
    It implies that $\psi$ satisfies a linear first order equation with a coefficient in an algebraic extension of $\mathcal K$, so that it is a Liouvillian function, which contradicts the hypothesis (by proposition \ref{T:44}, there is no Liouvillian solutions of \eqref{eq:linear}).
    \end{enumerate}

    \item[c)] Let us assume that $\mathcal H$ is of order $1$. Then ${\rm Lie}(\mathcal H_1)$ is determined by a linear differential equation of order one    
    whose solution space is contained in the solution space of \eqref{linear:kummer}. Let us consider then a solutions $\psi$ of this equation inside the Picard-Vessiot extension $\mathcal K\langle \psi_1^2,\psi_2^2,\psi_1\psi_2 \rangle$ of equation \eqref{linear:kummer}. We have a tower of Picard-Vessiot extensions,
    $$\mathcal K \subseteq \mathcal K\langle f \rangle \subset \mathcal K\langle \psi_1^2,\psi_2^2,\psi_1\psi_2 \rangle.$$
    As $\mathcal K \subseteq \mathcal K\langle f \rangle$ is Picard-Vessiot then the group of automorphisms of $\mathcal K\langle \psi_1^2,\psi_2^2,\psi_1\psi_2 \rangle$ fixing $ \mathcal K\langle f \rangle$ is a normal subgroup of ${\rm PSL}_2(\mathbb C)$. As ${\rm PSL}_2(\mathbb C)$ it can be either the identity or ${\rm PSL}_2(\mathbb C)$.
    \begin{enumerate}
    \item[c.i)] If it is the identity, then by the Galois correspondence the group of automorphisms of $\mathcal K\langle f \rangle$ over $\mathcal K$ should be isomorphic to ${\rm PSL}_2(\mathbb C)$. But this is impossible as the dimension of the Galois group, which is $3$, must coincide with the transcendence degree the extension, which is $1$. 
    \item[c.ii)] If it is ${\rm PSL}_2(\mathbb C)$ then by Galois correspondence $\psi \in \mathcal K$, but $f$ is a non-vanishing solution of \eqref{linear:kummer} that does not have Liouvillian solutions. 
    \end{enumerate}
    \end{enumerate}
Therefore $\mathcal G$ does not have any sub-$\mathcal D$-groupoid of order greater than $0$. 
\end{proof}

\subsection{Strong Minimality}

The notion of strong minimality comes from model theory and in the particular case of differential equations generalizes and unifies some classical notions of irreducibility. We say that an algebraic differential equation of third order \eqref{eq:es1} 
is \emph{strongly minimal} if and only if for every differential field  $\mathcal L$ (of which it is assumed that it includes the algebraic number) and every solution $j$ 
the transcendence degree ${\rm tr. deg.}_{\mathcal L} \mathcal L\langle j \rangle$ is necessary $0$ or $3$. That means no solution 
can satisfy a lower order equation unless it is an algebraic solution. A differential field  $\mathcal L$ for which there exists a solution  $j$ of \eqref{eq:es1} such that the transcendence degree of $\mathcal L\langle j \rangle$ over $\mathcal L$ is $1$ or $2$ is called a  \emph{witness of non-strong minimality}.

The proof of strong minimality of the Schwarzian equation \eqref{eq:es1} with $R(\lambda)\in \mathbb C(\lambda)$ can be found in \cite{casale2020axlindemannweierstrass}. 

\begin{proposition}[cf. \cite{casale2020axlindemannweierstrass} Th. 3.2]
Assume $R(\lambda)\in \mathbb C(\lambda)$. If the Riccati equation \eqref{riccati1} has no algebraic solutions, then the equation \eqref{eq:es1} is strongly minimal.  
\end{proposition}

Furthermore, this criterium is a necessary and sufficient condition. 

\begin{proposition}
If the Riccati equation \eqref{riccati1} has some algebraic solution, then the non-linear Schwarzian equation \eqref{eq:es1} is not strongly minimal.  
\end{proposition}

\begin{proof}
Suppose that the Riccati equation \eqref{riccati1} has an algebraic solution  $u(\lambda)$ in $\mathbb C(\lambda)^{\rm alg}$. In that case, automatically,  some solutions of the Schwarzian equation satisfy as well:

$$\frac{\lambda_{\tau\tau}}{(\lambda_\tau)^2}=2u(\lambda) $$

which is a second-order equation. Hence,  $\mathbb C(\lambda)^{\rm alg}$ is witnesses the non-strong minimality of the equation.
\end{proof}
 
Putting together these results with our main Theorem \ref{MR} and Proposition \ref{T:44} we obtain the following. 
  
\begin{proposition}\label{prop_final}
Assume $R(\lambda)\in \mathbb C(\lambda)$. The following statements are equivalent:
\begin{enumerate}
    \item Riccati equation \eqref{riccati1} has no algebraic solutions.
    \item Kummer groupoid $\mathcal G$, defined by the equation \eqref{kummer}, has no proper sub-$\mathcal D$-groupoids of order greater than $0$.
    \item The linear equation \eqref{eq:linear} has Galois group ${\rm SL_2}(\mathbb C)$ over $\mathbb C(X)$.
    \item The Schwarzian equation \eqref{eq:es1} is strongly minimal.
     \item The linear-Schwarzian equation \eqref{eq:es2} has no Liouvillian solutions.
\end{enumerate}
\end{proposition}

\begin{proof}
The only point we still need to check is the following. If there is an algebraic solution $u\in \mathbb C(X)^{\rm alg}$ of the Riccati equation \eqref{riccati1}, then $\mathcal G$ has some non trivial pseudogroup of order greater than zero. In order to clarify this point, let us see how solutions of the Ricatti equation \eqref{riccati1} allow us to reduce the linear-Schwarzian equation \eqref{eq:es2}. If $u(\lambda)$ is a solution of the Riccati equation, then let us look for a solution $\tau$ of the differential equation:

\begin{equation}\label{eq:affine} 
\frac{\tau_{\lambda\lambda}}{\tau_\lambda} = -2u(\lambda)
\end{equation}

We compute the Schwarzian derivative of $\tau$ obtaining:

$$S_{\lambda}(\tau) = (-2u(\lambda))_{\tau} - \frac{1}{2}(-2u(\lambda))^2 = -2(u'(\lambda) + u(\lambda)^2) = R(\lambda),$$

so that any solution of \eqref{eq:affine} is also a solution of \eqref{eq:es2}. Two any solutions of \eqref{eq:affine} with a
common domain of definition are related by an affine transformation of $\overline{\mathbb C}$. Thus, \eqref{eq:affine} is the differential equation
of an affine structure inside the projective structure determined by \eqref{eq:es2}. The inverses of the affine charts satisfy the differential equation,
\begin{equation}\label{eq:affine_lambda}
\frac{\lambda_{\tau\tau}}{\lambda_\tau^2} = 2u(\lambda).
\end{equation}
If we look for symmetries $\varphi$, such that for any solution $\lambda$ of \eqref{eq:affine_lambda}, the composition $\varphi\circ\lambda$ is also a solution, we arrive to the differential equation
\begin{equation}\label{eq:affine_sigma}
2u(\varphi)\varphi_{\lambda} = 2u(\lambda) + \frac{\varphi_{\lambda\lambda}}{\varphi_\lambda} .
\end{equation}
As $u$ is, in general, a transcendent function, equation \eqref{eq:affine_sigma} does not define an Zariski closed subset of ${\rm Aut}_2(X)$. However, if $u\in \mathbb C(X)^{\rm alg}$ then equation \eqref{eq:affine_sigma} defines a $\mathcal D$-groupoid of order $2$ contained in $\mathcal G$.
\end{proof}

It is interesting the criterium of strong-minimality for \eqref{eq:es1} coincides also with the simplicity of the $\mathcal D$-groupoid 
of its symmetries. It remains unclear to us a direct relation between the simplicity of the $\mathcal D$-groupoid and the strong minimality 
of the equation. It would be useful to know if the $\mathcal D$-groupoid can be used as a tool to detect strong minimality for some
other differential equations.

\bibliographystyle{plain}
\bibliography{biblio}

\end{document}